\documentclass[12pt]{article} 
\usepackage[T2A]{fontenc} 
\usepackage[utf8]{inputenc} 
\usepackage{comment, amsthm, amssymb, amsmath, mathrsfs, bbm, xcolor} 
\usepackage[draft]{hyperref}
\usepackage[english]{babel}
\newtheorem{conjecture}{Conjecture}

\newtheorem{theorem}{Theorem} 
\newtheorem{lemma}{Lemma}[section]

\newtheorem{remark}{Remark}

\usepackage{hyperref}

\newtheorem{proposition}{Proposition}

\newtheorem{rem}{Remark}[section]
\theoremstyle{definition}
\theoremstyle{plain}
\newtheorem{definition}{Definition}[section]
\renewcommand{\leq}{\leqslant}
\renewcommand{\geq}{\geqslant}

\DeclareMathOperator{\End}{End}

\begin{document}
\newcommand{\eps}{\varepsilon}
\renewcommand{\phi}{\varphi}

\newcommand{\suchthat}{\, : \,}
\newcommand{\commentDK}[1]{{\color{blue}{#1 -- DK}}}
\newcommand{\dima}[1]{\commentDK{#1}}
\newcommand{\fedya}[1]{{\color{orange}{#1 -- FP}}}
\newcommand{\red}[1]{{\color{red}{#1}}}
\newcommand{\green}[1]{{\color{green}{#1}}}

\author{Dmitry Krachun, Fedor Petrov}
\title{Tight lower bound on $|A+\lambda A|$ for algebraic integer $\lambda$}
\maketitle
\begin{abstract}
We prove an asymptotically tight lower bound on $|A+\lambda A|$ for $A\subset \mathbb{C}$ and algebraic integer $\lambda$. The proof combines strong version of Freiman's theorem, structural theorem on dense subsets of a hypercubic lattice and a generalisation of the continuous result on tight bound for
the measure of $K+\tau K$ for a compact subset
$K\subset \mathbb{R}^d$ of unit Lebesgue 
measure and a fixed linear operator $\tau\colon
\mathbb{R}^d\to \mathbb{R}^d$, obtained by the authors in \cite{KP}.
\end{abstract}

\section{Introduction}

For a subset $A\subset \mathbb{R}$ and a real number $\lambda\in \mathbb{R}$ we define the set $A+\lambda A$ to be
\[
A+\lambda A := \{a_1+\lambda a_2:\, a_1, a_2\in A\}.
\]
The question of finding the asymptotically minimal possible size of $A+\lambda A$ in term of the size of $A$ and $\lambda$ has received considerable attention over recent years. 

When $\lambda=p/q$ is a rational number with coprime
integers $p,q$, Bukh \cite{BUKH2008} proved that 
\[
|A+\tfrac{p}{q}A|\geq (|p|+|q|)\cdot |A|-o(A),
\]
and the error term was later improved to a constant $C=C(p, q)$ in the work of Balog and Shakan \cite{Balog2014}. This is the best possible up to the dependence of $C$ on $p, q$.

For transcendental $\lambda$ (it is easy to see that the bound does not depend on $\lambda$ in this case) the lower bound is no longer linear. Indeed, Konyagin and \L aba  \cite{Konyagin2006} showed that 
\[
|A+\lambda A|\geq C\frac{|A|\log |A|}{\log \log |A|}
\]
for an absolute constant $C$. 

This bound was then improved by Sanders \cite{Sanders2008}
to $|A|\log^{4/3-o(1)} |A|$, then by Schoen \cite{schoen2011} to $(\log{|A|})^{c\log\log{|A|}}|A|$ and again by Sanders \cite{sanders2012} to $e^{\log^c{|A|}}|A|$ for some $c>0$. All these bounds relied on the quantitative refinements of Freiman’s theorem. Very recently Conlon and Lim \cite{ConlonLim-transcendental} improved the bound to $e^{c\sqrt{\log{|A|}}}|A|$ for an absolute constant $c>0$, using much more elementary methods. This bound is tight up to value of the constant $c>0$, as follows from a construction from \cite{Konyagin2006}.

For the case of algebraic $\lambda$ we formulated a conjecture \cite[Conjecture 1]{KP} about the value of $\liminf |A+\lambda A|/|A|$ and proved the upper bound, see Conjecture \ref{conj:main-number} below. We also proved the conjecture for the specific case $\lambda = \sqrt{2}$. For the case $\lambda:=(p/q)^{1/d}$ this conjecture was proved by Conlon and Lim \cite{ConlonLim-algebraic}. In this paper we prove the conjecture for all algebraic integers $\lambda$.


To formulate the conjecture for arbitrary algebraic $\lambda$ we need the following  
\begin{definition}\label{def:H}
   For an irreducible polynomial $f(x)\in \mathbb{Z}[x]$ of degree $d\geq 1$ (irreducibility in particular means that the coefficients of $f$ do not have a common
integer divisor greater than 1) denote 
\[
H(f)=\prod_{i=1}^d (|a_i|+|b_i|),
\]
where $f(x)=\prod_{i=1}^d (a_ix+b_i)$
is a full complex factorization of $f$. 

For arbitrary polynomial $f(x)\in \mathbb{C}[x]$ we define $H(f)$ to be equal to 
$\min_{g|f} H(g)$, where the minimum is taken over all
irreducible polynomials $g(x)\in \mathbb{Z}[x]$ 
such that $g$ divides $f$ in $\mathbb{C}[x]$. In the case when $f$ has no non-constant divisors with integer coefficients we define $H(f):=\infty$. 

For a linear operator $\mathcal{T}\in \operatorname{End}(\mathbb{R}^d)$ we define $H^\circ(\mathcal{T})$ to be equal to $H(f)$, where $f$ is the characteristic polynomial of $\mathcal{T}$.

We also define, for a linear operator $\mathcal{T}\in\End(\mathbb{R}^d)$, $H(\mathcal{T})$ to be equal to $\prod_{i=1}^d (1+|\lambda_i|)$, where $\lambda_i$'s are the eigenvalues
of $\mathcal{T}$.
\end{definition}
Clearly, the value of $H(f)$ is well-defined, i.e. does not depend on the factorization. 


\begin{rem}
We note that $H(\mathcal{T})$ is not in general equal to $H^\circ(\mathcal{T})$. The former corresponds to the continuous problem of bounding the measure of $\Omega+\mathcal{T}\Omega$ for a set $\Omega\subset\mathbb{R}^d$ of measure 1, whereas the latter conjecturally corresponds to discrete problem of bounding $A+\mathcal{T}A$ for large sets $A$ of fixed size. In the case when $\mathcal{T}\in \operatorname{End}(\mathbb{Z}^d)$ has no invariant subspaces we have $H(\mathcal{T}) = H^\circ(\mathcal{T})$, see Proposition \ref{prop:h-h-0-relation}.
\end{rem}


With this definition, \cite[Conjecture 2]{KP} reads as

\begin{conjecture}\label{conj:main-operator}
Let  $\mathcal{T}\in \End(\mathbb{R}^d)$ be a linear operator. Then 
\[
\liminf_{|A|\rightarrow\infty,\, A\subset\mathbb{R}^d} \frac{|A+\mathcal{T} A|}{|A|} = H^\circ(\mathcal{T}).
\]
\end{conjecture}
This conjecture yields the following result for the behaviour of $\liminf |A+\lambda A|/|A|$ for algebraic $\lambda$, see \cite{KP} for details.
\begin{conjecture}\label{conj:main-number}
    Let $\lambda\in \mathbb{C}$ be an algebraic  number with minimal polynomial $f\in \mathbb{Z}[x]$. Then 
\[
\liminf_{|A|\rightarrow\infty,\, A\subset\mathbb{C}} \frac{|A+\lambda A|}{|A|} = H(f).
\]
\end{conjecture}

The main goal of the paper is to prove Conjecture \ref{conj:main-operator} for the case of $\mathcal{T}\in \End(\mathbb{Z}^d)$ and as a corollary, prove Conjecture \ref{conj:main-number} for the case of algebraic integer $\lambda$. We prove the following

\begin{theorem}\label{thm:main-operator}
Let  $\mathcal{T}\in\End(\mathbb{R}^d)$ be a linear operator such that $\mathcal{T}(\mathbb{Z}^d)\subset\mathbb{Z}^d$. Then 
\[
\liminf_{|A|\rightarrow\infty,\, A\subset\mathbb{R}^d} \frac{|A+\mathcal{T} A|}{|A|} = H^\circ(\mathcal{T}).
\]
\end{theorem}

\begin{theorem}\label{thm:main-number}
Let $\lambda\in \mathbb{C}$ be an algebraic integer number
with minimal polynomial $f$. Then
\begin{equation}\label{eq:algebraic}
    \liminf_{|A|\rightarrow\infty,\, A\subset\mathbb{C}} \frac{|A+\lambda A|}{|A|} = H(f)=
    \prod (1+|\lambda_i|),
\end{equation}
where $\lambda_i$'s are all algebraic conjugates of $\lambda$.
\end{theorem}


The rest of the paper is organised as follows. In Section 
\ref{sec:preliminary}
we relate $H^\circ(\mathcal{T})$ to $H(\mathcal{T})$ for $\mathcal{T}\in \End(\mathbb{Z}^d)$, deduce Theorem \ref{thm:main-number} from Theorem \ref{thm:main-operator}, and prove the upper bound in Theorem \ref{thm:main-operator}. Then in Section \ref{sec:reduction}, we reduce Theorem \ref{thm:main-operator} to the special case when the set $A$ is a subset of a $\mathbb{Z}^n\cap [0, N)^n$ of density at least $\eps=\eps(\mathcal{T})$. This reduction relies on a cirtain refinement of Freiman's theorem. Finally, in Section \ref{sec:box-dense}, we complete the proof of Theorem \ref{thm:main-operator} by using a structural lemma on the dense subsets of a hypercube, see Lemma \ref{lem:structural}, together with the continuous version of Theorem \ref{thm:main-operator}. 

\section{Preliminary observations}\label{sec:preliminary}

In this section we establish a relation between $H^\circ(\mathcal{T})$ and $H(\mathcal{T})$ for endomorphisms of the $\mathbb{Z}^d$ lattice, deduce Theorem \ref{thm:main-number} from Theorem \ref{thm:main-operator} and also prove an upper bound in Theorem \ref{thm:main-operator}.

\begin{proposition}\label{prop:h-h-0-relation}
Let $\mathcal{T}\in \End(\mathbb{Z}^d)$ then
\[
H^\circ(\mathcal{T}) = \min_{\alpha: \mathcal{T}\alpha\subset \alpha} H(\mathcal{T}_{|\alpha}),
\]
where the minimum is taken over all invariant subspaces $\alpha$ of $\mathbb{Q}^d$ and $\mathcal{T}$, as well as $\mathcal{T}_{|\alpha}$, is identified, with a slight abuse of notation, with its extensition to a linear operator on $\mathbb{Q}^d$ and $\mathbb{R}^d$. 
\end{proposition}
\begin{proof}
Let $f$ be the characteristic polynomial of $\mathcal{T}$. We first show that the minimum is at least $H^\circ(\mathcal{T})$. Take any invariant subspace $\alpha$, and let $g$ be the characteristic polynomial of $\mathcal{T}_{|\alpha}$. Clearly $g$ is a divisor of $f$, and taking rational basis of $\alpha$ and writing the matrix of $\mathcal{T}$ in this basis one sees that $g$ has rational coefficients. Moreover, since $\pm g$ is monic and all its roots are algebraic integers, it, in fact, has integer coefficients. So we have
\[
\min_{\alpha: \mathcal{T}\alpha\subset \alpha} H(\mathcal{T}_{|\alpha}) \geq \min_{g|f} H(g) =: H^\circ(\mathcal{T}).
\]

In the other direction, let $g$ be an irreducible divisor of $f$ of degree $m$ with integer coefficients. In particular, $\pm g$ is monic. Take the subspace $\beta:=\operatorname{Ker}(g(\mathcal{T}))$ of $\mathbb{Q}^d$ which is non-trivial since $g(\mathcal{T})$ is singular. Then take an arbitrary non-zero vector $v\in \beta$ and consider the subspace $\beta':=\langle v, \mathcal{T}v,\dots, \mathcal{T}^{m-1} v \rangle$ which is an invariant subspace of $\mathcal{T}$, as follows from the fact that $g(\mathcal{T}) v = 0$. Note that the characteristic polynomial of $\mathcal{T}_{|\beta'}$ is $\pm g$ since any eigenvalue of $\mathcal{T}_{|\beta'}$ is a root of $g$,  the dimension of $\mathcal{T}_{|\beta'}$ is at most $m$, and $g$ is monic. Hence, we have
\[
H^\circ(\mathcal{T}):=\min_{g|f} H(g)\geq \min_{\alpha: \mathcal{T}\alpha\subset \alpha} H(\mathcal{T}_{|\alpha}).
\]

\end{proof}
\begin{remark}\label{rem:smallest-invariant-subspace}
    Since $H(\mathcal{T}_{|\alpha})$ is non-decreasing in $\alpha$ under the partial ordering given by inclusion of subspaces, the invariant subspace $\alpha$ of minimal possible dimension among those with the minimal value of $H(\mathcal{T}_{|\alpha})$ additionally does not have any non-trivial invariant subspaces of $\mathcal{T}_{|\alpha}$. 
\end{remark}
\begin{proof}[Proof of Theorem \ref{thm:main-number} given Theorem \ref{thm:main-operator}]
The observation made in \cite[Lemma 2.1]{KP} implies that we may work with subsets of $\mathbb{Q}[\lambda]$, namely, that
\[
\liminf_{|A|\rightarrow\infty,\, A\subset\mathbb{C}} \frac{|A+\lambda A|}{|A|} = \liminf_{|A|\rightarrow\infty,\, A\subset\mathbb{Q}[\lambda]} \frac{|A+\lambda A|}{|A|} = \liminf_{|A|\rightarrow\infty,\, A\subset\mathbb{Z}[\lambda]} \frac{|A+\lambda A|}{|A|},
\]
where the last equality follows by dilating $A$. Since $\lambda$ is an algebraic integer, the linear operator $\mathcal{T}_\lambda$ defined by $x\mapsto \lambda x$ is an endomorphism of $\mathbb{Z}[\lambda]$. Furthermore, $\mathcal{T}_\lambda$ does not have invariant subspaces and its characteristic polynomial is equal, up to a sign, to the minimal polynomial $f$ of $\lambda$. So by Theorem \ref{thm:main-operator} we have 
\[
\liminf_{|A|\rightarrow\infty,\, A\subset\mathbb{Z}[\lambda]} \frac{|A+\lambda A|}{|A|} = H^\circ(\mathcal{T}_\lambda) = H(f), 
\]
as desired.
\end{proof}

Recall that a very similar argument is used 
in \cite[Proposition 1]{KP}.


\begin{proof}[Proof of the upper bound in Theorem \ref{thm:main-operator}] 
Using Remark \ref{rem:smallest-invariant-subspace}, choose a $\mathcal{T}$-invariant subspace $\alpha\subset \mathbb{Q}^d$ satisfying $H^\circ(\mathcal{T}) = H(\mathcal{T}_{|\alpha})$, and such that $\alpha$ has no non-trivial invariant subspaces of $\mathcal{T}_{|\alpha}$. Again, with some abuse of notation we identify $\alpha$ with a subspace of $\mathbb{R}^d$. 
It then suffices to construct large sets $A\subset \alpha$ such that $|A+\mathcal{T}_{|\alpha} A| \geq H(\mathcal{T}_{|\alpha})\cdot |A| - o(|A|)$, since $\mathcal{T}$ and $\mathcal{T}_{|\alpha}$ coinside on $\alpha$. So passing to $\mathcal{T}_{|\alpha}$ if needed, without loss of generality we may assume that the operator $\mathcal{T}$ itself has no non-trivial invariant subspaces and so $H^\circ(\mathcal{T}) = H(\mathcal{T})$.


Fix some small $\eps>0$. As explained after the proof of \cite[Theorem 2]{KP}, the inequality $\mu_\star(\Omega+\mathcal{T}\Omega)/\mu_\star(\Omega)
\geq H(\mathcal{T})$ is sharp and we can consider a convex compact set $\Omega\subset \mathbb{R}^d$ which satisfies $\mu(\Omega+\mathcal{T}\Omega)/\mu(\Omega)\leq H(\mathcal{T})+\eps$. Take $M$ large enough and consider the set $\Omega_M:=\mathbb{Z}^d\cap M\cdot \Omega$. We have $|\Omega_M|=\mu(\Omega)\cdot M^d+o(M^d)$, and since $\Omega_M+\mathcal{T}\Omega_M\subset (\Omega+\mathcal{T}\Omega)\cap\mathbb{Z}^d$ we have $|\Omega_M+\mathcal{T}\Omega_M|\leq \mu(\Omega+\mathcal{T}\Omega)\cdot M^d+o(M^d)$, where we used the fact that both $\Omega$ and $\Omega+\mathcal{T}\Omega$ are convex to approximate the number of integer points in their dilates. This immediately implies that $|\Omega_M+\mathcal{T}\Omega_M|/|\Omega_M|\leq H(\mathcal{T})+\eps+o_M(1)$, and since we can take $\eps$ to be arbitrary small, the upper bound follows.
\end{proof}

\section{Reduction to the case of a dense subset of a box}\label{sec:reduction}


To prove the lower bound in Theorem \ref{thm:main-operator} we first want to reduce the problem to the case of a set $A$ which forms a dense subset of a cube, i.e to the following statement
\begin{lemma}\label{lem:dense-subset}
    Let $\mathcal{T}:\mathbb{Z}^d\rightarrow\mathbb{Z}^d$ be a linear operator and $\eps>0$. For any subset $A\subset \{0, 1, \dots, N-1\}^d$ of size $|A|\geq \eps \cdot N^d$ we have 
    \[
    |A+\mathcal{T} A| \geq H^\circ(\mathcal{T})\cdot |A| - o(|A|),
    \]
    where the implied constant in $o(\cdot)$ may depend both on $\mathcal{T}$ and $\eps$.  
\end{lemma}

To deduce Theorem \ref{thm:main-operator} from Lemma \ref{lem:dense-subset} we need a strong version of Freiman's theorem which we now state and prove.


\subsection{Freiman's theorem}\label{subsec:Freiman}

\begin{definition}
Let $(G, +)$ be an abelian group. A set $P\subset G$ is a \emph{ generalised arithmetic progression}\footnote{Strictly speaking, a generalised arithmetic progressions is not just a set but the collection of data $(G; P; d; v_0,v_1,\dots,v_d;L_1,\dots L_d)$ but this would be cumbersome to write so with some abuse of notation we just write $P$ to denote this collection of data.} (GAP) of dimension $d\geq 1$ if it has the form
\begin{equation}\label{eq:GAP}
P=\left\{v_0 +\ell_1v_1+\dots+\ell_d v_d \suchthat 0\leq \ell_j \leq L_j \right\},
\end{equation}
where $v_0, v_1,\dots, v_d \in G, L_1, L_2, \dots, L_d\in \mathbb{Z}_+$. The generalised arithmetic progression $P$ is said to be \emph{proper} if all sums in \eqref{eq:GAP} are distinct (in which case $|P|=(L_1+1)(L_2+1)\dots (L_d+1)$). We say that $P$ is \emph{$k$-proper} if 
\begin{equation}\label{eq:k-GAP}
k\cdot P:=\left\{v_0 +\ell_1v_1+\dots+\ell_d v_d \suchthat 0\leq \ell_j \leqslant kL_j \right\},
\end{equation}
has all elements on the RHS distinct, i.e. if $P$ is proper and $|k\cdot P| = \prod_{j=1}^d(kL_j+1))$.
\end{definition}
It will be convenient for us to work with GAPs which are (almost) symmetric with respect to the origin. So we use the following 

\begin{definition}\label{def:centered-GAP}
    Let $(G, +)$ be an abelian group. We call a set $P\subset G$ is a \emph{centered generalised arithmetic progression} (c-GAP) of dimension $d\geq 1$ if it has the form
\begin{equation}\label{eq:c-GAP}
P=\left\{\ell_1v_1+\dots+\ell_d v_d \suchthat -L_j\leq \ell_j \leqslant L_j \right\},
\end{equation}
where $v_0, v_1,\dots, v_d \in G, L_1, L_2, \dots, L_d\in \mathbb{Z}_+$. For $k\geq 1$ and a centered GAP $P$ we write
\begin{equation}\label{eq:k-c-GAP}
k\star P:=\left\{\ell_1v_1+\dots+\ell_d v_d \suchthat -kL_j\leq \ell_j \leqslant kL_j \right\}.
\end{equation}
We say that $P$ is \emph{$k$-proper} if all elements on the RHS of \eqref{eq:k-c-GAP} are pairwise distinct. 
\end{definition}
\begin{rem}
    Note that any centered GAP $P$ can be seen as a GAP with $v_0 = -\sum L_j v_j$ and in this case $k\cdot P\not= k\star P$ for $k\geq 2$. However, the notion of being $k$-proper coincides for these two points of view, and this slight ambiguity should hopefully cause no confusion. 
\end{rem}



The following result is taken from \cite[Theorem 1.1]{Green2007}.
\begin{lemma}\label{lem:Freiman}

For every $K>0$ there exist constants $d=d(K)$ and $f=f(K)$ such that for any abelian group $G$ and any subset $A\subset G$ with doubling constant at most $K$ (i.e. such that $|A+A|\leq K|A|$) there exists a proper arithmetic progression $P\subset G$ containing $A$ which has dimension at most $d(K)$ and size at most $f(K)|A|$.
\end{lemma}

We need the following strengthening of this theorem, which ensures that not only $P$ itself is proper but also its large multiple is proper. Note that in this case we require the group $G$ to be torsion-free. 

\begin{lemma}\label{lem:Freiman-strong}
Let $\gamma:\mathbb{N}\times\mathbb{N} \rightarrow \mathbb{N}$ be an arbitrary function. For any $K>0$ there exist constants $d=d(K)$ and $F=F(K, \gamma)$ such that for any torsion-free abelian group $G$ and any subset $A\subset G$ with doubling constant at most $K$ (i.e. such that $|A+A|\leq K|A|$) there exists a generalized arithmetic progression $P\subset G$ containing $A$ which has dimension at most $d(K)$, size at most $F(K, \gamma)|A|$, and is $k$-proper with $k:=\gamma(\lfloor|P|/|A|\rfloor, d(P))$, where $d(P)$ is the dimension of $P$.
\end{lemma}
\begin{proof}
We use \cite[Theorem 3.40]{TaoWu2006} which states that any $d$-dimensional GAP $P$ in a torsion-free abelian group $G$ can be embedded in a proper GAP $P'$ of size at most $d^{C_0d^3}|P|$ for fixed constant $C_0$, and that if $P$ is non-proper, then $P'$ can be taken to have dimension at most $d-1$. Note a caveat that in \cite[Theorem 3.40]{TaoWu2006} this latter statement about the decrease in the dimension is stated for any abelian group $G$ but it, in fact, only holds, and is proved, for the torsion-free case. 

Now, we prove the lemma with the same $d(K)$ as in Lemma \ref{lem:Freiman}. 
First, consider a proper arithmetic progression $P_0$ of dimension $d_0 \leq d(K)$ and size at most $f(K)|A|$ which contains $A$. 
If $P_0$ is $k_0$-proper with $k_0:=\gamma(\lfloor |P_0|/|A|\rfloor, d)$ we stop. Otherwise, consider a GAP $P_1\supset k_0\cdot P_0$ of dimension $d_1\leq d_0-1$ and size at most $d^{C_0d^3}|k_0P_0|$. 
Again, if $P_1$ is $k_1$-proper with $k_1:=\gamma(\lfloor |P_1|/|A| \rfloor, d_1)$ we stop, otherwise we consider $P_2\supset k_1\cdot P_1$ of dimension  $d_2\leq d_1-1$ and size at most $d^{C_0d^3}|k_1P_1|$, etc. After some $s\leq d(K)$ steps we stop and obtain a GAP $P_s$ of dimension $d_s\leq d$ which is $k_s$ proper with $k_s:=\gamma(\lfloor |P_s|/|A|\rfloor, d_s)$. Moreover, $|P_s|/|A|$ is bounded by a function which only depends on $d(K), F(K)$ and $\gamma$.

\end{proof}


\subsection{Lemma \ref{lem:dense-subset} implies Theorem \ref{thm:main-operator}}

Recall that it only remains to prove the lower bound in Theorem \ref{thm:main-operator}. Before proving the reduction to the case of a dense subset of a cube, we observe that we may assume that $A\subset \mathbb{Q}^d$, see Lemma \ref{lem:rational-set-reduction}, and then show that if $|A+\mathcal{T}A| \ll |A|$ then the set $A\cup \mathcal{T}A$ can be embedded in a centered generalised arithmetic progression $P$ which is $k$-proper for some large $k$, see Lemma \ref{lem:shift-in-centred-GAP}. We also prove a simple lemma which is then used in the proof of the reduction.

\begin{lemma}\label{lem:rational-set-reduction}
The lower bound in Theorem \ref{thm:main-operator} follows from the lower bound in the special case when $A\subset \mathbb{Q}^d$.    
\end{lemma}
\begin{proof}
Take an arbitrary finite set $A\subset \mathbb{R}^d$ write all relations of the form $a+\mathcal{T}b=c+\mathcal{T}d$ for $(a,b,c,d)\in A^4$ which are satisfied in coordinates. This gives a system of homogeneous linear equations over $\mathbb{Q}$. 
Together with all conditions ensuring that all points of $A$ are distinct 
(for any two points $a\ne b\in A$ 
we take a condition of non-equality type 
$a_j\ne b_j$ for certain coordinate index $j$)
this gives us a system 
of equalities and non-equalities that, since solvable over $\mathbb{R}$ (by elements of $A$) is also solvable over $\mathbb{Q}$ giving us a set $A'\subset \mathbb{Q}$ for which $|A'+\mathcal{T}A'|\leq |A+\mathcal{T}A|$.

\end{proof}

 So from now on we assume that $A\subset \mathbb{Q}^d$ and using induction we further assume that the statement has been proved for all operators in dimensions $1, 2, \dots, d-1$.

\begin{lemma}\label{lem:shift-in-centred-GAP}
    For any function $\gamma:\mathbb{N}\times\mathbb{N} \rightarrow \mathbb{N}$, an operator $\mathcal{T}\in \operatorname{End}(\mathbb{Z}^d)$, and $K>0$ there exist constants $n_0=n_0(\mathcal{T}, K)$ and $F=F(\gamma, \mathcal{T}, K)$ such that the following holds. Assume that $A\subset \mathbb{Z}^d$ satisfies $|A+\mathcal{T}A|\leq K\cdot |A|$. Then there exists a centred generalised arithmetic progression $P$ in $\mathbb{Z}^d$ of dimension $n\leq n_0$ which has size at most $F\cdot |A|$, contains $(A-x)\cup \mathcal{T}(A-x)$ for some $x\in\mathbb{Z}^d$, and is $k$-proper with $k:=\gamma(\lfloor|P|/|A|\rfloor, n)$.
 \end{lemma}
\begin{proof}
Without loss of generality we may assume that $\gamma$ is increasing in each variable. By Pl\"unnecke inequality the set $B:=A\cup \mathcal{T}A$ satisfies $|B+B|\leq (2K^2+K)\cdot |B|$ and so by Lemma \ref{lem:Freiman-strong} we can embed both $A$ and $\mathcal{T}A$ in some GAP $P$ of dimension $n=O_K(1)$ and size $O_K(|A|)$ which is $\gamma'(\lfloor|P|/|A|\rfloor, n):=2\cdot \gamma(2^{n}\cdot \lfloor|P|/|A|\rfloor, n)$ proper. 

Take arbitrary $x\in A$ and consider $A':=A-x$. Since $A, \mathcal{T}A, \{x\}, \{\mathcal{T}x\} \subset P$, we have $A', \mathcal{T}A'\subset P-P =: P'$ which is a centred GAP. Also $P'$ has size at most $4^n|P|$ and is $k/2$ proper whenever $P$ is $k$-proper. It remains to note that $\gamma(\lfloor|P'|/|A'|\rfloor, n) \leq \gamma'(\lfloor|P|/|A|\rfloor, n) / 2$ by the definition of $\gamma'$.
\end{proof}

\begin{lemma}\label{lem:linear-combination-of-vectors}
    Let $v_1,\dots, v_d, v\in \mathbb{Z}^n$ be vectors such that $v\in \langle v_1,\dots, v_d\rangle$. 
    Assume that all coordinates of all $v_j$ are bounded, in absolute value, by some constant $C$. Then there exist integers $s, s_1,\dots, s_d$ such that 
    \[
    sv=\sum_{j=1}^d s_j v_j
    \]
    where $s_j=O_{C, n}(\|v\|_{L^\infty})$ for each $j=1,\dots, d$ and $s=O_{C, n}(1)$.
\end{lemma}
\begin{proof}
Take some minimal subset $S$ of vectors among $v_1,\dots, v_d$ which linearly span $v$. Then vectors from $S$ are linearly independent and so we can augment them with several vectors of the standard basis of $\mathbb{Q}^n$ to form a basis $S'$ of $\mathbb{Q}^n$. It then remains to consider the unique linear combination of vectors in $S'$ giving $v$. All vectors that we added to $S$ will come with zero coefficients and so we will obtain a linear combination of vectors in $S$ giving $v$ in which all coefficients are rational numbers with denominators of size $O_{C, n}(1)$ and numerators of size $O_{C, n}(\|v\|_{L^\infty})$ as changing one basis to another multiplies the vector of coefficients by some fixed matrix with entries having bounded numerators and denominators.
\end{proof}

\begin{proof}[Proof of Theorem \ref{thm:main-operator} given Lemma \ref{lem:dense-subset}]


By Lemma \ref{lem:rational-set-reduction} and dilating $A$ if necessary, we assume that $A\subset \mathbb{Z}^d$. We also induct on the dimension $d$ assuming that statement has been proved for all smaller dimensions. Note that for the base case $d=1$ equivalence of Theorem \ref{thm:main-operator} and Lemma \ref{lem:dense-subset} immediately follows from Freiman's theorem. 

For a large finite set $A\subset \mathbb{Z}^d$ we want to show that 
    \[
    |A+\mathcal{T}A| \geq H^\circ(\mathcal{T})\cdot |A| + o(|A|).
    \]

In proving this we may assume the contrary, so $|A+\mathcal{T}A|\leq K\cdot |A|$ with $K=K(\mathcal{T})=H^\circ(\mathcal{T})$. Since $|A+\mathcal{T}A|=|(A-x)+\mathcal{T}(A-x)|$ for every $x\in\mathbb{Z}^d$, using Lemma \ref{lem:shift-in-centred-GAP} we may assume that both $A$ and $\mathcal{T}A$ are inside some centred generalised arithmetic progression $P$ of dimension $n$ which is $k$-proper with  $k:=\bold{k}(\lfloor |P|/|A|\rfloor, n)$ with function $\bold{k}$ to be defined later, and such that $|A|/|P|\geq \eps=\eps(\bold{k}, \mathcal{T})$.




Let $w$ be the basis vector of $P$ corresponding to the largest $L_j$. For this vector we know that $w, 2w, \dots, Lw\in P$ where $L:=L_j\gg |A|^{1/n}$. Since $L\gg |A|^{1/n}$, we may assume that $L$ is large enough in terms of $\mathcal{T}$ for our argument to work. Consider vectors $w_0:=w, w_1:=\mathcal{T}w, w_2:=\mathcal{T}^2w,\dots, w_{d-1}:=\mathcal{T}^{d-1}w$. We consider two cases depending on whether these vectors are linearly independent or not.

\textbf{Case 1: } Vectors $w_0, w_1,\dots, w_{d-1}$ are linearly dependent in $\mathbb{Q}^d$. Then the hyperplane $\alpha$ spanned by these vectors has dimension smaller than $d$ and is such that the set $A$ lies in at most $O(|A|/L)=O(|A|^{1-1/n})$ translates of $\alpha$. So we may write $A=A_1\sqcup A_2\sqcup\dots \sqcup A_m$ where $A_i\subset x_i+\alpha$, and $m\leq C(\mathcal{T})\cdot |A|^{1-1/n}$.
We may assume that $\lambda = -1$ is not an eigenvalue of $\mathcal{T}$, as otherwise $H^\circ(\mathcal{T}) = 2$ and the inequality $|A+\mathcal{T}A|\geq H^\circ(\mathcal{T})\cdot |A| - 1$ follows from 
the torsion-free version of
Cauchy–-Davenport theorem. 
Note that for $i\not = j$ we have $(A_i+\mathcal{T}A_i)\cap (A_j+\mathcal{T}A_j) = \emptyset$ as otherwise we would have $[\operatorname{Id}+\mathcal{T}](x_i-x_j) \in \alpha$ and since $\operatorname{Id}+\mathcal{T}$ is invertible and $\mathcal{T}\alpha \subset \alpha$ this would imply $x_i-x_j\in \alpha$ contradicting the fact that translates $x_i+\alpha$ and $x_j+\alpha$ are distinct.

For the sets $B_i:=A_i-x_i \subset \alpha$ we have $|A_i+\mathcal{T}A_i| = |B_i+\mathcal{T}_{|\alpha}B_i|$ and so the lower bound for the operator $\mathcal{T}_{|\alpha}$, which has dimension smaller than $d$, and sets $B_1,\dots, B_m$ gives us 
\[
|A+\mathcal{T}A| = \sum_{j=1}^m |B_j+\mathcal{T}_{|\alpha}B_j|
\geq
\sum_{j=1}^m H^\circ(\mathcal{T}_{|\alpha})\cdot |B_j| - o\left(\sum_{j=1}^m |B_j|\right) - O(m),
\]
where the last term comes from all the sets $B_j$ of constant size. Since $\sum |B_j| = |A|$ and $m\ll |A|^{1-1/n}$ this immediately implies the result for the operator $\mathcal{T}$ since $H^\circ(\mathcal{T}_{|\alpha})\leq H^\circ(\mathcal{T})$ by Proposition \ref{prop:h-h-0-relation}.

\textbf{Case 2:} Vectors $w_0, w_1,\dots, w_{d-1}$ are linearly independent in $\mathbb{Q}^d$. We proceed in several steps:

\textbf{Step 1:} We show that there exist some constants $\lambda=\lambda(\mathcal{T}, d, n, \eps)$ and $k'=k'(\mathcal{T}, d, n, \eps)$ such that all vectors $\ell \lambda w_{i}$ with $i\in \{0, 1, \dots, d-1\}$ and $\ell\in [0\dots L]$ are going to be in a multiple $k'\star P$ of $P$.

It suffices to show the existence of such $\lambda_j$ and $k_j$ for each $w_j$ separately and then take $\lambda:=\prod \lambda_j$ and $k':=\lambda\cdot\max_j\{k_j\}$. We induct on $j$. For $w_0:=w$ this follows from the construction as we took $w$ to be the basis vector with the largest coordinate $L$. Now, assume that for $w_{j-1}$ we have some values $\lambda_{j-1}, k_{j-1}$. Then $\lambda_{j-1} w_{j-1}, 2\lambda_{j-1} w_{j-1}, \dots, L \lambda_{j-1} w_{j-1}$ are all in $k_{j-1}\star P$. We want to show that there exist some $\lambda_j, k_j$ such that $\lambda_j s w_j \in k_j\star P$ for any $s\in [1,\dots, L]$. To prove this, note that it is sufficient to find $\lambda_j, k_j$ such that $\lambda_j s w_j \in k_j\star P$ holds for any $s\in [1,\dots, \delta L]$ with some constant $\delta=\delta(\lambda_{j-1}, k_{j-1},\eps, n)>0$ and then multiply $k_j$ by $\lceil 1/\delta \rceil$ to cover all $s\in [1,\dots, L]$. Indeed, this follows from a trivial observation that any $s\in [1,\dots L]$ can be written as a sum of at most $\lceil 1/\delta\rceil$ summands, each of which is in $[1,\dots, \delta L]$, and the fact that we have $L$ large enough. We now show how to construct such $\lambda_j, k_j$ for $\delta:=\eps/(k_{j-1}+1)^n$, where we recall that $\eps$ is a lower bound for $|A|/|P|$.

To this end, choose arbitrary $s\in [1, \dots, \delta L]$ and let $w=w(s):=sw_{j-1}$. Consider the following shifts of $A$: 
\[
A, A + \lambda_{j-1} w, A+2\lambda_{j-1} w, \dots A+ \frac{(k_{j-1}+1)^n}{\eps}\cdot \lambda_{j-1} w.
\] 
Since $A\subset P$, by induction hypothesis and the fact that $s\cdot \tfrac{(k_{j-1}+1)^n}{\eps} \leq L$, all these sets are in $P+k_{j-1}\star P$. As each of these sets has size $|A|\geq \eps|P|$ and the set $P+k_{j-1}\star P$ has size smaller than $(1+k_{j-1})^n |P|$, by Dirichlet's principle two of the sets must intersect and so we have, for some $c = c(s)\leq \tfrac{(k_{j-1}+1)^n}{\eps}$, that $c\lambda_{j-1}w\in A-A$. Since $\mathcal{T}(A-A)=\mathcal{T}A-\mathcal{T}A\subset P-P\subset 2\star P$, this implies that $c\cdot s\lambda_{j-1} w_j = \mathcal{T}(c\lambda_{j-1}w) \in 2\star P$. Which implies that for $C:=\operatorname{lcm}(1,\dots,\lfloor\tfrac{(k_{j-1}+1)^n}{\eps}\rfloor)$ we have
\[
C\cdot s\lambda_{j-1}w_j \in 2C\star P.
\]
Since $s\in [1,\dots, \delta L]$ was arbitrary, we can take $\lambda_j:=C\lambda_{j-1}$ and $k_j:=2C$. As mentioned above, to cover all $s\in [1,\dots, L]$ it is then sufficient to multiply $k_j$ by $\lceil 1/\delta\rceil$. This completes the proof of the induction step.

\textbf{Step 2: } We show that, for $\lambda$ and $k'$ as above, all $\lambda\cdot w_j$'s (which are in $k'\star P$)
have only small coordinates in the basis of $P$, and all non-zero coordinates correspond to dimensions with $L_j\gg L$. Indeed, write $\lambda w_j=\sum_{s=1}^n x_s v_s$, where $v_s\in [-k'L_s, k'L_s]$. Then with $t:=1+ \min_s k'\cdot L_s/|v_s|$ 
(where the minimum is taken over the coordinates with $v_s\not= 0$) we have $t \lambda w_j \in (2k')\star P\setminus k'\star P$ by the fact that $P$ is at least $2k'$ proper. This implies that we must have $t>L$ which in turn implies that $|k'L_s/v_s| > L$ for each coordinate $s$ where $v_s\not=0$. Since $L=\max L_j$, this is only possible if $|v_s|\leq k'$ and $L_s\geq L/k'$, proving the claim.


\textbf{Step 3:} Now, consider a natural embedding $\iota$ of $P$ into $\mathbb{Z}^n\subset \mathbb{Q}^n$ (i.e. $\iota$ maps basic vectors of $P$ to the standard basis of $\mathbb{Z}^n$) which by properness can be extended to $k\star P$, and consider a linear subspace $\alpha \subset \mathbb{Q}^n$ spanned by $\{\iota\lambda w_0, \dots, \iota \lambda w_{d-1}\}$. Split $\iota A\subset \iota P$ into subsets given by the intersections with shifts of this linear subspace. We claim that the corresponding subdivision $A:=A_1\sqcup \dots \sqcup A_m$ satisfies $(A_i+\mathcal{T}A_i)\cap (A_j+\mathcal{T}A_j) =\emptyset$ for $i\not= j$. 

Indeed, arguing from contradiction, we assume that $y_i+\mathcal{T} x_i=y_j+\mathcal{T} x_j$ for certain $x_i,y_i
\in A_i$ and $x_j,y_j\in A_j$. Then $\iota (\operatorname{Id}+\mathcal{T})(x_i-x_j) = \iota(x_i-y_i-x_j+y_j) \in \alpha$ and $\iota (x_i-x_j)\not\in \alpha$.

Since $x_i, x_j\in A\subset P$ and also $\mathcal{T} A\subset P$, we have  $z:=(\operatorname{Id}+\mathcal{T})(x_i-x_j) \in P+P-P-P = 4\star P$. We also know that $\iota z\in\alpha$ and $\iota z$ has all coordinates in $\mathbb{Z}^n$ at most $4L$ in absolute value. Since $\iota \lambda w_j \in \mathbb{Z}^n$ has all coordinates of size $O_{\mathcal{T}, \eps}(1)$ for each $j=0, 1, \dots, d-1$ by the argument in the second step of the proof, we can apply Lemma \ref{lem:linear-combination-of-vectors} to deduce that
$\iota \lambda' z = \sum_{j=0}^{d-1} s_j\iota w_j $ with $s_j=O_{\mathcal{T}, \eps}(L)$ and $\lambda'=O_{\mathcal{T}, \eps}(1)$. Since $\iota$ is well defined on $k\star P$, choosing the function $\bold{k}$ correctly this ensures that $\lambda' z=\sum_{j=0}^{d-1} s_j w_j$.

Now, as $(\operatorname{Id}+\mathcal{T})^{-1}$ can be written as a polynomial of $\mathcal{T}$ with integer coefficients of size 
$O_{\mathcal{T}}(1)$, call it $f(\mathcal{T})$, we have  
\[\lambda'(x_i-x_j) = f(\mathcal{T})[\lambda'(\operatorname{Id}+\mathcal{T})(x_i-x_j)] = f(\mathcal{T})[\lambda' z]
\]
Reducing $f(x)\cdot (\sum_j s_jx^j)$ modulo the minimal polynomial of $\mathcal{T}$, we can rewrite the latter expression as a linear combination of $w_0,\dots, w_{d-1}$ with integer coefficients of size $O_{\mathcal{T}, \eps}(L)$. Again, assuming function $\bold{k}$ was chosen large enough, this linear combination is in $k\star P \cap \iota^{-1}\alpha$, and so this implies that $\iota (x_i-x_j)\in \alpha$, giving a contradiction. This completes the proof of the fact that $A_i\cap \mathcal{T}A_i$ are pairwise disjoint.

\textbf{Step 4:} Second step in this proof ensures that in the representation $A:=A_1\sqcup \dots \sqcup A_m$ we have $m =O_{\mathcal{T},\eps}(|P|/L^d)$, and so it suffices to prove that $|A_j+\mathcal{T}A_j|\geq H^\circ(\mathcal{T})\cdot |A_j|-o(L^d)$ for each $j=1,2,\dots, m$ and then sum all these inequalities. After shifting $A_j$ by some $x_j\in A_j$ we have $\iota (A_j-x_j)\subset \alpha \cap [-2L, 2L]^n\subset \mathbb{Z}^n$ and by Lemma \ref{lem:linear-combination-of-vectors} and the argument in the second step of this proof we know that for some constants $\lambda''=O_{\mathcal{T}, \eps}(1)$ and $L_0=O_{\mathcal{T},\eps}(L)$ we have that \[
\lambda''(A_j-x_j)\subset \left\{\sum_{k=0}^{d-1} m_kw_k\,\mid\, m_k\in [-L_0, L_0]\right\},
\] 
which gives us a natural linear map $\iota': \lambda''(A_j-x_j)\rightarrow [-L_0, L_0]^d\subset \mathbb{Z}^d$ such that $\iota' \mathcal{T} = \widetilde{\mathcal{T}}\iota$ with the operator $\widetilde{\mathcal{T}}$ acting on the standard basis $\{e_0,\dots e_{d-1}\}$ of $\mathbb{Z}^d$ as $e_i\mapsto e_{i+1}$ for $i=0,\dots, d-2$ and $e_{d-1}\mapsto \sum_{k=0}^{d-1} \alpha_k e_k$ where $x^n-\sum_{k=0}^{d-1}\alpha_k x^k$ is the characteristic polynomial of $\mathcal{T}$. Let $B_j$ be the image of $A_j-x_j$ in $\mathbb{Z}^d$ under $\iota'$. Since $H^\circ(\widetilde{\mathcal{T}})=H^\circ(\mathcal{T})$ we have, by Lemma \ref{lem:dense-subset},
\[
|A_j+\mathcal{T}A_j|=|\lambda''(A_j-x_j)+\mathcal{T}[\lambda''(A_j-x_j)]|=|B_j+\widetilde{\mathcal{T}}B_j|\geq H^\circ(\mathcal{T})|B_j|+o(L^d),
\]
where we note that the last inequality is trivially true if $|B_j|=o(L^d)$.


\end{proof}

\section{The case of a dense subset of a box
} \label{sec:box-dense}

In this section we prove Lemma \ref{lem:dense-subset}. In order to do so, we approximate a discrete set $A\subset [0, N)^d\cap \mathbb{Z}^d$ by a continuous density function and then use the following generalisation of \cite[Theorem 2]{KP} 


\begin{lemma}\label{lem:density-functions-inequality}
    Let $\mathcal{T}\in \operatorname{End}(\mathbb{R}^d)$, and $K\subset \mathbb{R}^d$ be a compact set. Assume that measurable non-negative functions $f:K\rightarrow\mathbb{R}_+$ and $h:K+\mathcal{T}K\rightarrow\mathbb{R}_+$ satisfy, for any $x, y\in K$, the inequality $h(x+\mathcal{T}y)\geq f(x)$. Then one has
    \[
        \int_{K+\mathcal{T}K} h(z)\, d\mu(z)\geq H(\mathcal{T})\cdot \int_{K}f(x)\, d\mu(x),
    \]
    where $H(\mathcal{T}):=\prod_{i=1}^d (1+|\lambda_i|)$ with $\lambda_i$'s being eigenvalues of $\mathcal{T}$, and $\mu$ is the Lebesgue measure on $\mathbb{R}^d$. 
\end{lemma}
\begin{rem}
  \cite[Theorem 2]{KP} bounds the volume
  of $K+\mathcal{T}K$ from below as $H(\mathcal{T})$ times the volume of $K$. In other words,
  it exactly coincides with the case $f=\mathbbm{1}_K$ and $h=\mathbbm{1}_{K+\mathcal{T}K}$. 
\end{rem}
\begin{proof}
    Consider the set $K_f\subset\mathbb{R}^{d+1}$ defined by $K_f:=\{(x, t):\, x\in K, \, 0\leq t\leq f(x)\}$, and the operator $\mathcal{T}'\in \operatorname{End}(\mathbb{R}^{d+1})$ defined by $\mathcal{T}'(x, t):=(\mathcal{T}(x), 0)$. Then the inequality $h(x+\mathcal{T}y)\geq f(x)$ implies the inclusion 
    \[
    K_f+\mathcal{T}'(K_f)\subset (K+\mathcal{T}K)_h,
    \]
    and so it suffices to apply \cite[Theorem 2]{KP} to the set $K_f$ of measure equal to $\int_K f(x)\, d\mu(x)$ and the map $\mathcal{T}'$ which satisfies $H(\mathcal{T}')=H(\mathcal{T})$.
\end{proof}





To approximate a discrete set $A$ by a continuous density function, we need the following structural result. In the following, for an integer $M$ by an \emph{$M$-cube} we mean a cube $[0, M)^d$ shifted by an element of $(M\mathbb{Z})^d$.



\begin{lemma}\label{lem:structural}
For any $\eps, \delta > 0$ and $d\geq 1$ there exists $B_0=B_0(\eps, \delta, d)$ such that the following holds. Let $\mathcal{P}:=[0, N)^d \cap \mathbb{Z}^d$ be a cube and let $A\subset \mathcal{P}$ be a set of size at least $\eps|\mathcal{P}|$. Then there exist $B \leq B_0$ and a collection $\{\mathcal{P}_1,\dots,\mathcal{P}_s\}$ of disjoint $N/B$-cubes such that the set $A':=A\cap (\cup \mathcal{P}_i)$ satisfies 
\begin{itemize}
\item $|A'|\geq (1-\delta) |A|$ 
\item $A'$ is topologically $\delta$-dense in each $\mathcal{P}_i$, in the sense that $\forall x\in \mathcal{P}_i\, \exists y\in A'\cap \mathcal{P}_i:\, |x_j-y_j|_{L^\infty}\leq \delta N/B $ for $j=1,\dots, d$.
\end{itemize}
\end{lemma}


\begin{proof}

In the following we tacitly assume $1/\delta$ to be an integer. For $\ell\geq 0$ let $B_\ell:=\delta^{-\ell}$. For each $\ell$ split $\mathcal{P}$ into $B_\ell^d$ equal parts and let $\mathcal{P}^{(\ell)}:=\cup_{i=1}^{s_\ell}\mathcal{P}_i^{(\ell)}$ be the union of parts which contain at least one point of $A$. By construction we have $\mathcal{P}^{(\ell)}\subset \mathcal{P}^{(\ell-1)}$.

Notice that $|\mathcal{P}^{(\ell)}|\geq |A|\geq \eps\cdot |\mathcal{P}|$ and so for some $\ell\leq \tfrac{\log{\eps}}{\log{(1-\delta^{d+1}\eps)}}$ we must have $|\mathcal{P}^{(\ell+1)}|\geq (1-\delta^{d+1}\eps)|\mathcal{P}^{(\ell)}|$. This means that at least $1-\eps\delta$ fraction of $\mathcal{P}_i^{(\ell)}$'s are subsets of $\mathcal{P}^{(\ell+1)}$ (i.e. we kept all $\delta^{-d}$ smaller parts of them). Let $\mathcal{Q}^{(\ell)}\subset \mathcal{P}^{(\ell)}$ be the union of such $\mathcal{P}_i^{(\ell)}$'s and define $A':=A\cap \mathcal{Q}^{(\ell)}$. Then the second condition of the lemma is satisfied and we also have
\[
|A'|\geq |A|-|\mathcal{P}^{(\ell)}\setminus\mathcal{Q}^{(\ell)}|\geq |A|-\eps\delta\cdot|\mathcal{P}^{(\ell)}|\geq |A|-\eps\delta|\mathcal{P}|\geq (1-\delta)|A|.
\]
Moreover, by construction we have $B = \delta^{-\ell}\leq \exp{\{\tfrac{-\log{\eps}\cdot \log{\delta}}{\log{(1-\delta^{d+1}\eps)}}\}}$
\end{proof}

\begin{rem}
    By throwing away additionally at most $\eps|A|$ points from $A'$ we could ask the density of $A'$ to be at least $\eps^2/2$ in each of the $\mathcal{P}_i$. 
\end{rem}

We now turn to proving Lemma \ref{lem:dense-subset} which we restate for the convenience of the reader.

\newtheorem*{lem:dense-subset}{Lemma \ref{lem:dense-subset}}

\begin{lem:dense-subset}
    Let $\mathcal{T}:\mathbb{Z}^d\rightarrow\mathbb{Z}^d$ be a linear operator and $\eps>0$. For any subset $A\subset \{0, 1, \dots, N-1\}^d$ of size $|A|\geq \eps \cdot N^d$ we have 
    \[
    |A+\mathcal{T} A| \geq H^\circ(\mathcal{T})\cdot |A| - o_N(|A|),
    \]
    where the implied constant in $o_N(\cdot)$ may depend both on $\mathcal{T}$ and $\eps$.   
\end{lem:dense-subset}

\begin{proof}[Proof of Lemma \ref{lem:dense-subset}]

Let $\delta>0$ be small enough. Since $A\subset [0, N)^d$, we have $\mathcal{T}A\subset [-CN, CN]^d$ for some $C=C(\mathcal{T})$. Using Lemma \ref{lem:structural} we construct $\mathcal{T}A'\subset \mathcal{T}A$ of size at least $(1-\delta )|A|$ and a collection $\{\mathcal{Q}_1,\dots, \mathcal{Q}_s\}$ of cubes of size $N/B$, where $B < B_0(\eps, \delta, d, C)$, such that $\mathcal{T}A'$ is topologically $\delta$-dense in each of $\mathcal{Q}_i$'s. 

Now, consider a set $\mathcal{T}^{-1}\left(\cup \mathcal{Q}_i\right)$ and approximate it with a collection of $\delta'N/B$-cubes $\{\mathcal{P}_1,\dots, \mathcal{P}_{s'}\}$ by taking all $\delta'N/B$-cubes inside each of the sets $\mathcal{T}^{-1}\mathcal{Q}_j$. For any $\delta'$ small enough in terms of $\mathcal{T}$ we can ensure that 
\[
|\mathcal{Q}_j\setminus (\cup_{j'=1}^{s'} \mathcal{P}_{j'})| \leq C_1(\mathcal{T})\cdot \delta' \cdot |\mathcal{Q}_j|, 
\]
with some constant $C_1(\mathcal{T})$ depending only on $\mathcal{T}$.

We now consider the set $K:=\cup_{i=1}^{s'} \mathcal{P}_i$ and a piece-wise constant function $f:K\rightarrow\mathbb{R}_+$ defined on it by $f(x):=|A\cap \mathcal{P}_i|/|\mathcal{P}_i|$ for each $x\in \mathcal{P}_i$. We then cover $K+\mathcal{T}K$ by $\delta'N/B$ cubes $\{\mathcal{R}_1,\dots, \mathcal{R}_{s''}\}$ and consider a piece-wise constant function $h:K+\mathcal{T}K\rightarrow\mathbb{R}_+$ defined by $h(z):=|(A+\mathcal{T} A)\cap \mathcal{R}_i|/|\mathcal{R}_i| + (1+\delta/\delta')^d-1$ for each $z\in \mathcal{R}_i$.


\bigskip

\textbf{Claim:} Functions $f, h$ satisfy the assumption of Lemma \ref{lem:density-functions-inequality}, i.e for any $x\in K$ and $y\in K$ one has $h(x+\mathcal{T} y) \geq f(x)$. 

\textbf{Proof of the claim:} Indeed, consider a $\delta'N/B$-cubes $\mathcal{P}_{j'}$ containing $x$ and $\mathcal{R}_{j''}$ containing $x+\mathcal{T}y$. For some $y_0\in (\delta'N/B \cdot \, \mathbb{Z})^d$ we have $\mathcal{R}_{j''} = y_0 + \mathcal{P}_{j'}$. Since $x+\mathcal{T}y$ lies both in $\mathcal{P}_{j'}+\mathcal{T}y$ and $\mathcal{R}_{j''} = y_0 + \mathcal{P}_{j'}$, we must have $\|y_0-\mathcal{T}y\|_{L^\infty} \leq \delta' N / B$.

By construction we have $\mathcal{T}K\subset \cup \mathcal{Q}_i$, so there exists some $j\in [1, \dots, s]$ such that $\mathcal{T}y\in \mathcal{Q}_j$. The fact that $y_0\in (\delta'N/B \, \mathbb{Z})^d$ and $\|y_0-\mathcal{T}y\|_{L^\infty} \leq \delta' N / B$ implies that $y_0$ also lies in (the closer of) $\mathcal{Q}_j$. Since $\mathcal{T}A'$ is $\delta$-dense in $\mathcal{Q}_j$ there exists $y'\in A'$ such that $\|\mathcal{T}y'-y_0\|\leq \delta N/B$. We then have 
\begin{equation}\label{eq:h-bound}
|(A+\mathcal{T}A)\cap \mathcal{R}_{j''}|
\geq 
|((A\cap \mathcal{P}_{j'})+\mathcal{T}y')\cap \mathcal{R}_{j''}|
\geq 
|A\cap \mathcal{P}_{j'}| - |(\mathcal{P}_{j'}+\mathcal{T}y')\setminus \mathcal{R}_{j''}|
\end{equation}
Since $\mathcal{R}_{j''} = y_0+\mathcal{P}_{j'}$ and $\|\mathcal{T}y'-y_0\|\leq \delta N/B$ we can bound the last term by 
\[
(\delta'N/B + \delta N / B)^d - (\delta' N / B)^d \leq |\mathcal{R}_{j''}|\cdot \left((1+\delta/\delta')^d-1\right).
\]
Dividing \eqref{eq:h-bound} by $|\mathcal{R}_{j''}| = |\mathcal{P}_{j'}|$ we infer that $h(x+\mathcal{T}y)\geq f(x)$. This concludes the proof of the claim.

\bigskip

Now by Lemma \ref{lem:density-functions-inequality} we know that 
\begin{equation}\label{eq:inequality-densities}
\int_{K+\mathcal{T}K} h(z)\, d\mu(z)\geq H(\mathcal{T})\cdot \int_{K}f(x)\, d\mu(x).    
\end{equation}
 
Recalling the definition of $h$ we can upper bound the LHS by 
\[
|A+\mathcal{T}A| + ((1+\delta/\delta')^d-1)\cdot |K+\mathcal{T}K| \leq |A+\mathcal{T}A| + ((1+\delta/\delta')^d-1) \cdot |A|/\eps \cdot C_2(\mathcal{T}).
\]
Whereas for the integral on the right we have a lower bound of 
\begin{align*}
|A'| - |\mathcal{T}^{-1}(\cup \mathcal{Q}_j) \setminus (\cup \mathcal{P}_{j'})| 
&\geq
(1-\delta)|A| - C_1(\mathcal{T})\cdot \delta'^{1/d}\cdot |\cup \mathcal{Q}_j|
\\&\geq 
|A|-\delta|A| - C_3(\mathcal{T})\cdot (|A|/\eps) \cdot \delta'^{1/d},
\end{align*}
where for the last inequality we used the fact that $|A|/\eps \geq N^d$ and that all cubes $\mathcal{Q}_j$ are inside $[-CN, CN]^d$ for $C=C(\mathcal{T})$. It then remains to choose first $\delta'$ small enough in terms of $\mathcal{T}$ and $\eps$ and then $\delta$ small enough in terms of $\delta', \eps, \mathcal{T}$ to conclude that \eqref{eq:inequality-densities} implies that 
\[
|A+\mathcal{T}A|\geq H(\mathcal{T})\cdot |A| - o(|A|).
\]

\end{proof}

\textbf{Acknowledgement: } We would like to thank Ilya Losev for useful discussions and useful comments on earlier version of the paper.

\bibliographystyle{plain}

\bibliography{main}

\end{document}